\documentclass[11pt]{article}

\usepackage{amsmath}
\usepackage{amsthm,amscd,amsfonts}
\usepackage{amssymb, upref, color}
\usepackage{amsmath,amssymb,amsthm,mathrsfs}
\usepackage{color}
\usepackage[dvips]{graphicx}
\usepackage{epsf,epsfig, verbatim} 
\usepackage{latexsym,bm}
\usepackage{indentfirst}
\usepackage{color}
\usepackage{cite}

\usepackage{epsfig}
\usepackage{epstopdf}
\usepackage{subfigure}
\usepackage[font=scriptsize]{caption}
\usepackage{graphicx}
\graphicspath{{figures/}}

\numberwithin{equation}{section}

\theoremstyle{plain}
\newtheorem{exam}{Example}[section]
\newtheorem{theorem}[exam]{Theorem}

\newtheorem{remark}[exam]{Remark}

\begin{document}
	\sloppy
	\captionsetup[figure]{labelfont={bf},name={Fig.},labelsep=period}
	\captionsetup[table]{labelfont={bf},name={Table},labelsep=period}

	\title{Constant vorticity
		two-layer water flows in the  $\beta$-plane
		approximation with centripetal forces\footnote{  This paper was jointly supported from Natural Science Foundation of Zhejiang Province (No.LZ23A010001), National Natural Science Foundation of China (No. 11671176, 11931016).}	}
	\author
	{
		Yuchao He$^{a}$\,\,\,\,\,Yongli Song$^{b}$\,\,\,\,
		Yonghui Xia$^{a}\footnote{Corresponding author.   yhxia@zjnu.cn;xiadoc@163.com }$
		\\
		{\small \textit{$^a$ School of Mathematical  Science,  Zhejiang Normal University, 321004, Jinhua, China}}\\
		{\small Email:  YuchaoHe@zjnu.edu.cn; xiadoc@163.com; yhxia@zjnu.cn}
		\\
		{\small \textit{$^b$ School of Mathematics,  Hangzhou Normal University,  311121,  Hangzhou, China}}\\
		{\small Email: songyl@hznu.edu.cn}
	}
	\maketitle
	\begin{abstract}
		The  constant vorticity {\bf two-layer water wave} in the $\beta$-plane  approximation with centripetal forces is investigated in this paper. Different from the works (Chu and Yang\cite[JDE, 2020]{chu} and Chu and Yang \cite[JDE,  2021]{chu2}) on the singe-layer wave flows, we consider the  two-layer water wave model containing
		a free surface and an interface.  The interface separates two layers with different features such as velocity field, pressure and vorticity.	 We prove that if the change in pressure in the $y$-axis direction is bounded, then the pressure is a function only related to depth and the surfaces of the water flows. And the inner wave will not affect the pressure function, if the water flow densities in each layer are equal. Furthermore,  the explicit expressions of the velocity, pressure   are given for the two-layer water flows. It is interesting that our method and results are also valid for the multi-layer water waves. Let the number of layers of water waves $n$ tend to infinity,
		we prove that the sequence of pressure in  the lowest layer    $\{P_1^n(x,y,z,t)\}_{n\geq1}$ is uniformly  convergent, if the density of each layer is bounded and the each surface of wave flows is uniformly convergent.
	\end{abstract}

	\section{Introduction}
	In recent years, water wave problems have received widespread attention and research from scholars, especially three-dimensional water wave problems. Three-dimensional water waves have extremely complex physical characteristics (see  \cite{book,cons2014,Iooss} ), which are usually described by nonlinear partial differential equation, and it is certainly one of the major challenges of the
	next years. An interesting problem is to study the dynamic behavior of three-dimensional water waves with Coriolis forces. Coriolis forces make the water wave problem highly complex in both mathematics and physics (see Constantin and Johnson  \cite{A. Cons} and Constantin\cite{cons-book}). It, along with vorticity and stratification, affects the dynamic behavior of water waves (see Wheeler \cite {wheeler}). To address such problems, the $f$-plane approximation and $\beta$-plane approximation are effective means.
	 The  pioneering mathematical study was initiated by Constantin \cite{cons1}, he found a solution for equatorial water waves in $\beta$-plane approximation.
	 Martin  \cite{martin1} studied the singe-layer  equatorial
	 wave near the  in the $\beta$-plane approximation.
	  Constantin and Johnson\cite{A. Cons} presented a meaningful work in the $\beta$-plane approximation.
	 Chu and Yang \cite{chu} considered centripetal force based on Martin's \cite{martin1} model. As an astonishing result, they  proved that
	 if centripetal forces exist and the vorticity is constant, then the surface of equatorial flow is actually flat and the vorticity is equal to 0. Subsequently, the results at any latitude are presented in Chu and Yang's  recent work \cite{chu2}.
	
Vorticity is quite important in describing the dynamic behavior of water waves. The study of water waves with vorticity has a long history, dating back to the work of Gerstner \cite{Gerstner} in  19th century.
Subsequently, a large number of scholars conducted in-depth research on rotational water waves  (see Constantin and Escher \cite{cons3} and  Constantin et al. \cite{cons2,cons4}. Especially, some studies on constant vorticity have achieved good results. Constantin \cite{cons2011} proved that the free surface water flow with constant non-zero vorticity below the wave train and above the plate is two-dimensional. After Constantin's work, Martin \cite{martin2018} proved if the vorticity is not equal to $0$, then  time-dependent 3D gravity water flows does not exist. Henry \cite{henry1} studied the water wave equation under the Coriolis force in the $\beta$-plane approximation and  the exact solution   is presented.
Groves and Wahl$\acute{e}$n \cite{Wahlen} studied the existence of small-amplitude Stokes and solitary gravity water waves  with an arbitrary distribution of vorticity. Henry \cite{henry2016} promoted the research on geophysical fluid dynamics by presenting the solution to  a $\beta$-plane approximation of the water wave equation with Coriolis and centripetal forces for the equatorial current.
Wang et al. \cite{wang1,wang2} applied the water wave theory to solve atmospheric problems and obtained the  solution of Ekman flow in the $f$-plane approximation  and $\beta$-plane approximation.
 Later, Martin \cite{martin2023}  extended the previous work \cite{martin2018} to the two-layer  water waves and obtained the Liouville-type results. Henry \cite{Henry2} studied the underlying fluid motion in two-layer water flows. The latest work on more general vorticity can be referenced in Chu and Escher\cite{c-e}, Chu et al. \cite{chu3}, Ionescu-Kruse \cite{IONESCU-KRUSE}, Dai and Zhang \cite{z-d} and Basu and Martin \cite{basu}.

Different from the previous works of Martin \cite{martin2023} and Henry\cite{Henry2} on the  two-layer water waves, we study the dynamical behavior of  constant vorticity two-layer water waves in  {\bf $\beta$-plane approximation}.   We obtained the solution of two-layer water waves and proved that the pressure is independent of the stratification of water waves. And a comparison  with Chu and Yang's work \cite{chu} of single layer  equatorial currents is presented. We do not need to assume that the wave surface has a specific traveling wave form, and the pressure function we obtain is a function related to depth and water wave surfaces. And our results can be extended to multi-layer water waves. Let the number of layers of water waves $n$ tend to infinity,
we prove that the squence of pressure in  the lowest layer    $\{P_1^n(x,y,z,t)\}_{n\geq1}$ is uniformly  convergent, if the density of each layer is bounded and the each surface of wave flows is uniformly convergent.

	\section{Double-layer three-dimensional water wave equations}

	It is reasonable to imagine the Earth as a perfect sphere with radius $R=6378$ km (see \cite{chu}). And the Earth's rotational speed is approximately constant at $\Omega = 7.29 \times10^5$ rad/s. Let the $x$-axis be direction of horizontal water flow due east, the $y$-axis horizontally
	due north and the $z$-axis vertically upward. Under the above definition, the water flows model in the $\beta$-plane approximation with Coriolis term and centripetal
	forces is described by following governing equations
	\begin{equation}\label{original}
		\begin{cases}
		u_{t}+uu_{x}+vu_{y}+wu_{z}+2\Omega w-\beta yv=-\frac{P_{x}}{\rho},\\
		v_{t}+uv_{x}+vv_{y}+wv_{z}+\beta  yu+\Omega^2y=-\frac{P_{y}}{\rho},\\
		w_{t}+uw_{x}+vw_{y}+ww_{z}-2\Omega u-\Omega^2R=-\frac{P_{z}}{\rho}-g,
		\end{cases}
	\end{equation}
with the mass conservation condition:
\begin{equation}\label{mass}
u_{x}+v_{y}+w_{z}=0.
\end{equation}
 Since equations \eqref{original} and \eqref{mass} hold both in the lower and upper layers,  we temporarily omit the subscript representing the number of layers in the following analysis.
	By taking $\frac{\widetilde{P}}{\rho}=\frac{P}{\rho}-\frac{\Omega^2}{2}y^2+(\Omega^2R-g)z$, equation \eqref{original} is transformed into
	\begin{equation}\label{transform}
		\begin{cases}
			u_t+uu_x+vu_y+wu_z+2\Omega w-\beta yv=-\frac{\widetilde{P}_{x}}{\rho},\\
		v_t+uv_x+vv_y+wv_z+\beta yu=-\frac{\widetilde{P}_{y}}{\rho},\\
		w_t+uw_x+vw_y+ww_z-2\Omega u=-\frac{\widetilde{P}_{z}}{\rho}.
		\end{cases}
	\end{equation}
	In this paper, we approximate that vorticity is fixed in every layer of water waves (the vorticity in different layers can be unequal) and   denoted as
\begin{equation}\label{wodu}
\Lambda=(\Lambda_1,\Lambda_2,\Lambda_3)=(w_y-v_z,u_z-w_x,v_x-u_y).
\end{equation}
Combining \eqref{transform} with  \eqref{wodu}, we obtain
	\begin{equation}\label{line-PDE}
		\begin{cases}
		\Lambda_1u_x+(2\Omega+\Lambda_2)u_y+(\Lambda_3+\beta y)u_z=0,\\
		\Lambda_1v_x+(2\Omega+\Lambda_2)v_y+(\Lambda_3+\beta y)v_z=0,\\
		\Lambda_1w_x+(2\Omega+\Lambda_2)w_y+(\Lambda_3+\beta y)w_z-\beta v=0.
		\end{cases}
	\end{equation}
And we assume $\Lambda_2+2\Omega\ne0$ in the  subsequent study.

The three-dimensional water wave problem we are studying is actually equivalent to the linear partial differential equations problem above. For this problem, we  solve it by using the characteristic equation method in the next section.

For the double-layer water wave problems, the situation is more complex (see figure 1). Although double-layer water waves come into contact with each other on adjacent internal wave surfaces, different layers of water waves typically exhibit completely different dynamic behaviors (see \cite{martin2023}).
In this paper, we assume the pressure is balanced at the interface of water waves in different layers.
In other words, for the case of double-layer water waves, the pressure of the upper wave $P_2$ is equal to the pressure $P_1$ of the lower wave on the inner wave surface. That is $
P_1(x,y,\eta_1(x,y,t),t)=P_2(x,y,\eta_1(x,y,t),t).$

\begin{figure}[h]
	\centering
	\includegraphics[width=10cm]{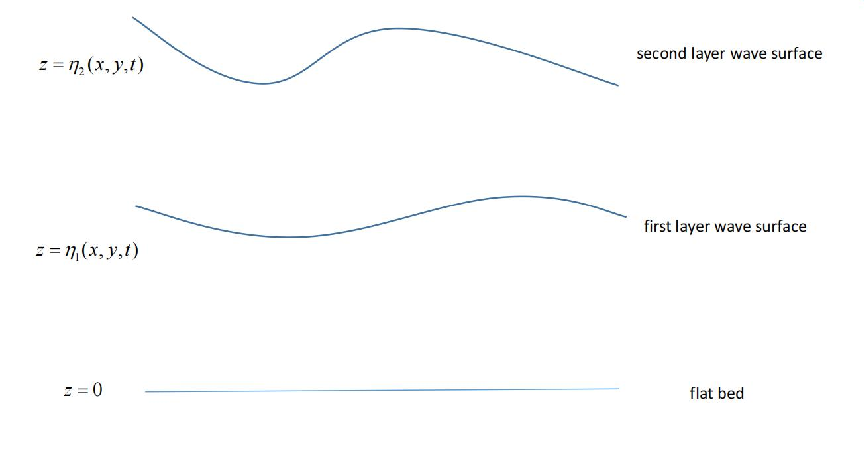}
	\caption{Schematic diagram of two-layer water waves.}
	\label{figure}
\end{figure}

\section{Main results}
We first give the results on the lower fluid flow in the case of  two-layers water flows.
\begin{theorem}\label{theorem 1}
In double-layer water waves, 	assume pressure  $P_{1y}$ in lower water wave and inner wave surface $\eta_{1y}$ are bounded.  Then the solution of \eqref{original} in the lower water waves is
	\[
		u_1=-\frac{\Omega^2}{\beta}
	,\ v_1=0,\ w_1=0,
	\]
	and
	\[ P_1=\rho_1[(\frac{2\Omega^3}{\beta}-\Omega^2R+g)z+c_1(t)],\]
where $c_1$ is a time dependent function.
	\end{theorem}
\begin{proof}
Based on the assumption of the vorticity  \eqref{wodu}, we obtain
\[
w_{1zx}=u_{1zz}\quad \rm{and} \quad u_{1yy}=v_{1zy}.
\] 
According to the mass conservation condition \eqref{mass}, we have
\[
\Delta v_1=v_{1xx}+v_{1yy}+v_{1zz}=(u_{1x}+v_{1y}+w_{1z})_y=0. 
\]
Analogously, we obtain 
\[\Delta u_1=\Delta w_1=0,\]
which means $u_1$, $v_1$ and $w_1$ are all harmonic functions in the lower fluid domain. Moreover $w_{1z}$ is also harmonic function.\\
According to the third equation of \eqref{line-PDE}, we have
\[
\Delta(yw_{1z})=0,
\]
i.e.
\[
y\Delta w_{1z}+2w_{1zy}=0.
\]
From this, we  immediately obtain
\[
w_{1zy}=0.
\]
Analogously, $v_{1zy}=u_{1zy}=0$.\\
Moreover, according to \eqref{wodu}, we obtain
\[
u_{1zy}=w_{1xy}=v_{1zy}=w_{1yy}=0.
\]
Based on the above analysis, we conclude that $w_{1y}$ is independent of $x$, $y$ and $z$.
Combining with the  boundary conditions \[
w_1\equiv 0  \quad\ \rm{on} \quad z=0,
\]
we infer that 
\[
w_{1y}\equiv0.
\]
From \eqref{wodu}, we immediately conclude that $v_{1z}=-\Lambda_1$.\\
Then by differentiating with respect to $y$ in the third equation of \eqref{line-PDE}, we  obtain
\[
\Lambda_1w_{1xy}+(\lambda_2+2\Omega)w_{1yy}+(\Lambda_3+\beta y)w_{1zy}+\beta w_{1z}-\beta v_{1y}=0,
\]
Moreover, we have $w_{1z}=v_{1y}$,
which implies \[
w_{1zz}=v_{1yz}=(-\lambda_1)_y=0,
\]
and $v_{1yy}=w_{1zy}=0.$
Note that $v_1$ and $w_1$ are harmonic functions ($\Delta v_1=\Delta w_1=0$). \\
We can infer that\[
v_{1xx}=w_{1xx}=0.
\]
Combining with \eqref{wodu}, we obtain that
\[
u_{1xz}=w_{1xx}=0.
\]
Analogously,  $u_{1xy}=v_{1xx}=0$.
Moreover, by differentiating with respect to $z$, we have
\[
\Lambda_1u_{1xz}+(\Lambda_2+2\Omega)u_{1yz}+(\Lambda_3+\beta y)u_{1zz}=0,
\]
which implies 
\begin{equation}\label{uzz}
(\lambda_3+\beta y)u_{1zz}=0.
\end{equation}
Note that \eqref{uzz} holds within the lower fluid domain. Then $u_{1zz}=0$ holds in the lower fluid domain, when $y\ne -\frac{\Lambda_3}{\beta}.$ And  by the continuity of $u_{1zz}$ in the lower fluid domain, then $u_{1zz}=0$ holds within the lower fluid domain.

It is not difficult to obtain that $w_{1xz}=0$ by the definition of the  vorticity  \eqref{wodu}.

Since $w_{1xy}=w_{1xx}=0$, we obtain that $w_{1x}$ only depends on the time $t$. Note that the boundary condition $w_1\equiv 0$ on $z=0$. Then $w_{1x}=0$ holds within the lower fluid domain.
And we immediately conclude that
$w_{1zx}=0$ and $\Lambda_2=u_z$.

Combining with $w_{1zy}=w_{1zz}=0$, we conclude that $w_{1z}$ is independent of $x$, $y$ and $z$.
Differentiating the first equation of \eqref{line-PDE} and noting that $v_{1zx}=(-\Lambda_1)_x=0$, $v_{1xx}=0$, we obtain 
\[
(\Lambda_2+2\Omega)v_{1yx}=0
\]
Based on the assumption $\Lambda_2+2\Omega\ne0$,  $v_{1yx}=0$ holds within the fluid domain.\\
Moreover, we have 
\[u_{1yy}=v_{1yx}=0.
\]
Combining the first and second equations of \eqref{line-PDE} with  $u_{1xy}=u_{1yy}=v_{1xy}=v_{1yy}=0$, we obtain
\begin{equation}\label{uz}
u_{1z}=v_{1z}=0\ \quad \rm{within\ the\ fluid\ domain}, 
\end{equation}
which implies
\[
\Lambda_1=\Lambda_2=0.
\]
Based on\eqref{line-PDE}, we conclude that 
\begin{equation}\label{uy}
u_{1y}=v_{1y}=0 \ \quad \rm{within\ the\ fluid\ domain}.
\end{equation}
Moreover, based on \eqref{mass}, we have $w_{1z}=0.$
Then combining with the third equation of \eqref{line-PDE}, we obtain 
\[
v=0 \ \quad \rm{within\ the\ fluid\ domain}.
\]
 By the definition of the  vorticity, we have 
 \[
 \Lambda_3=v_{1x}-u_{1y}=0.
 \]
It is easy to see that
 \[
u_x=-v_y-w_z=0.
 \]
 Combining this with \eqref{uy} and \eqref{uz}, we have  $u_1=a_1(t)$ for some function $a_1.$
 And according to \eqref{transform}, we obtain
 \begin{equation}
 	\begin{cases}
 		\widetilde{P}_{1x}=-\rho_1 a_1'(t),\\
 		\widetilde{P}_{1y}=-\rho_1\beta y a_1(t),\\
 		\widetilde{P}_{1z}=2\rho_1\Omega a_1(t).
 	\end{cases}
 \end{equation}
 Moreover,
 $$\widetilde{P}_1=\rho_1\left(-a_1'(t)x-\frac{\beta}{2}a_1(t)y^2+2\Omega a_1(t)z+c_1(t)\right),$$
 and
 $$P_1=\rho_1\left(-a_1'(t)x+(\frac{\Omega^2}{2}-\frac{\beta}{2}a_1(t))y^2+(2\Omega a_1(t)-\Omega^2R+g)z+c_1(t)\right).$$
 On the free surface, we have
 \[
 [P_{1y}+P_{1z}\eta_{1y}]\bigg|_{z=\eta_1(x,y,t)}=\rho_1[\left(\Omega^2-\beta a_1(t)\right)y+\left(2\Omega a_1(t)-\Omega^2R+g\right)\eta_{1y}].
 \]
 And we  limit the variation of $P_1$ in the $y$-direction to be bounded. That means there exists a $M>0$ such that 
 \[
 -M<[P_{1y}+P_{1z}\eta_{1y}]\bigg|_{z=\eta_1}  <M
 \]
 holds for any $y$ within the fluid domain.
 Since $a_1(t)$ and $\eta_{1y}$ is finite, then  $u_1=a_1(t)=\frac{\Omega^2}{\beta}.$
 We immediately obtain
 \begin{equation}\label{P_1}
 	P_1=\rho_1[(\frac{2\Omega^3}{\beta}-\Omega^2R+g)z+c_1(t)],
 \end{equation}
 which is very natural in physics.
\end{proof}

\begin{remark}
Unlike Chu and Yang's work \cite{chu}, we do not require the wave surface to have a traveling wave form $\eta_1(x,y,t)=\eta_1(x-ct,y).$ And the pressure we obtain is only related to depth and the surfaces of the water flows, which is shown in the following theorem. By the method of   Theorem \ref{theorem 1}, we can prove the main results  in Chu and  Yang's work  on the single-layer water wave (Theorem 3.2, \cite{chu}).
\end{remark}

\begin{theorem}

In the two-layer water wave, we assume that $P_{2y}$, $\eta_{2y}$ are bounded and the velocity vector $w$  is  continuous on the  internal wave $z=\eta_1(x,y,t)$. 
If the condition in Theorem \ref{theorem 1} still holds, then
\[
P_2=k\rho_2(z- \eta_2(x,y,t))+P_{atm}
\]
and
\[
P_1=\rho_1kz+(\rho_2-\rho_1)k\eta_1(x,y,t)-\rho_2k\eta_2(x,y,t),
\]
where $k=\frac{2\Omega^3}{\beta}-\Omega^2R+g.$
\end{theorem}
\begin{proof}
According to the chain rule, we obtain
\[
P_x(x,y,\eta_1(x,y,t),t)=\eta_{1x}P_z|_{z=\eta_1(x,y,t)}.
\]
Since the speed velocity $w$ is continuous  on the bottom of upper flow, we have
\[
w_2(x,y,\eta_1,t)=w_1(x,y,\eta_1,t)=0.
\]
By  the similar analysis in the proof of Theorem \ref{theorem 1}, we obtain
\[
P_1=\rho_1\left[-a'_1(t)x+(\frac{\Omega^2}{2}-\frac{\beta}{2}a_1(t))y^2+(2\Omega a_1(t)-\Omega^2R+g)z+c_1(t)\right],
\]
\[
P_2=\rho_2\left[-a'_2(t)x+(\frac{\Omega^2}{2}-\frac{\beta}{2}a_2(t))y^2+(2\Omega a_2(t)-\Omega^2R+g)z+c_2(t)\right].
\]

Note that $P_{1y}$, $P_{2y}$ and $\eta_{1y}$, $\eta_{2y}$ are bounded, we obtain $a_2(t)=\frac{\Omega^2}{\beta}.$

Since $P_1(x,y,\eta_1(x,y,t),t)=P_2(x,y,\eta_1(x,y,t),t),$
we  immediately get
\[
 c_2(t)=\frac{\rho_1}{\rho_2}(k\eta_1(x,y,t)+c_1(t))-k\eta_1.
\]
Moreover,
\[
P_2(x,y,z,t)=\rho_2kz+c_2(t).
\]
According to the boundary condition $P_2=P_{atm}$ on $z=\eta_2(x,y,t)$, we obtain
\[
P_{atm}=\rho_2k\eta_2(x,y,t)+c_2(t).
\]
Through the boundary condition on the free surface $z=\eta_2(x,y,t)$,  we obtain the pressure functions:
\[
P_2=\rho_2k(z- \eta_2(x,y,t))+P_{atm},
\]
and
\[
 P_1=\rho_1kz+(\rho_2-\rho_1)k\eta_1+P_{atm}-\rho_2k\eta_2.
\]
\end{proof}

\begin{remark}
	Assume the densities are equal in each layer of water flows ($\rho_1=\rho_2$). Then  according to the continuity of pressure during the inner wave surface, we obtain $P_1=(\frac{2\Omega^3}{\beta}-\Omega^2R+g)(z- \eta_2(x,y,t))+P_{atm}.$ That implies that the pressure inside the entire water flow is only related to the depth of the water flow and the surface of the upper water flow.
	\end{remark}
\begin{remark}
It is worth mentioning that our results can be generalized to the $n$-layer water wave model (see Fig. 2) by induction method.
The pressure of the $i$-th layer water flows can be expressed as
\begin{equation}\label{P_i}
P_i=\rho_ikz+\sum_{j=i}^{n-1}k\rho_{j+1}(\eta_j(x,y,t)-\eta_{j+1}(x,y,t))-k\rho_i\eta_i+P_{atm}.
\end{equation}
If the number of layers of water waves $n$ tends to infinity, 
we obtain the following inequality by \eqref{P_i}:
\[
|P_1^{n+p}(x,y,z,t)-P_1^{n}(x,y,z,t)|\leq k\sup_{j\geq_1}{\rho_j}\sum_{j=n}^{n+p}|\eta_j(x,y,z,t)-\eta_{j+1}(x,y,z,t)|.
\]
Moreover, $\{P_1^n(x,y,z,t)\}_{n\geq1}$ is uniformly convergent, if $\{\rho_i\}_{i\geq1}$ is bounded and
 $\{\eta_i(x,y,t)\}_{i\geq1}$ is uniformly convergent.

\end{remark}

\begin{figure}[h]
	\centering
	\includegraphics[width=10cm]{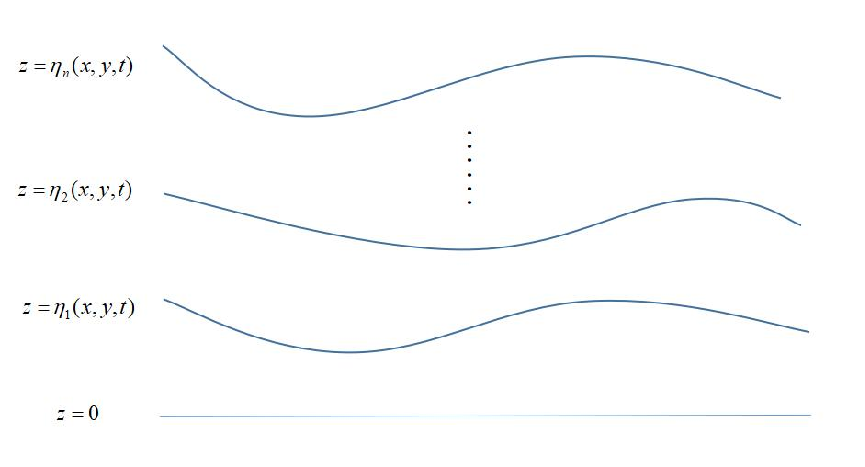}
	\caption{Schematic diagram of $n$-layer water waves.}
	\label{figure}
\end{figure}

\begin{remark}
	For an $n$-layer  water wave with a flat surface at the top layer (i.e. $\eta_n(x,y,t)=d_0$), we assume the densities are equal in each layer of the water flows $(\rho_1=\rho_2=\cdots=\rho_n).$  Then we obtain $c_1(t)=P_{atm}-(\frac{2\Omega^3}{\beta}-\Omega^2R+g)d_0.$ And in the entire water flow,  $P=(\frac{2\Omega^3}{\beta}-\Omega^2R+g)(z-d_0)+P_{atm}$. The pressure inside the water wave is independent of the number of layers of waves and the internal wave surface.
	\end{remark}

\section*{Declaration of Interests.}
 The authors report no conflict of interest.


\begin{thebibliography}{999}
	\bibitem{book}
	R. Johnson, A Modern Introduction to the Mathematical Theory of Water Waves, Cambridge University Press, 1997.




\bibitem{cons2014}
A. Constantin, Some nonlinear, equatorially trapped, nonhydrostatic
internal geophysical waves,  J. Phys. Oceanogr., 44(2)  (2014) 781-789.

\bibitem{Iooss}
G. Iooss and P. Plotnikov, Asymmetrical three-dimensional travelling
gravity waves, Arch. Ration. Mech. Anal., 200(3) (2011)  789-880.


\bibitem{A. Cons}
A. Constantin, R.S. Johnson, An exact, steady, purely azimuthal equatorial flow with a free surface, J. Phys.
Oceanogr., 46 (2016) 1935–1945.


\bibitem{cons-book}
A. Constantin, Nonlinear Water Waves with Applications to Wave-Current Interactions and Tsunamis, CBMS-NSF
Conference Series in Applied Mathematics, vol. 81, SIAM, Philadelphia, 2011.


\bibitem{wheeler}
M. Wheeler, On stratified water waves with critical layers and Coriolis forces, Discrete Contin. Dyn. Syst., 8 (2019)
4747–4770.


\bibitem{cons1}
A. Constantin, On the modelling of equatorial waves, Geophys. Res. Lett., 39 (2012) L05602.



\bibitem{martin1}
C. Martin, On constant vorticity water flows in the $\beta$-plane approximation, J. Fluid Mech., 865 (2019) 762–774.
	
	
	
	\bibitem{chu}
	J. Chu, Y. Yang, Constant vorticity water flows in the equatorial $\beta$-plane
	approximation with centripetal forces, J. Differential Equations, 269 (2020) 9336–9347.
	
	

	
		\bibitem{chu2}
	J. Chu, Y. Yang, A cylindrical coordinates approach to constant vorticity
	geophysical waves with centripetal forces at arbitrary
	latitude, J. Differential Equations, 279 (2021) 46-62.
	
		\bibitem{Gerstner}
	F. Gerstner, Theorie der Wellen samt einer daraus abgeleiteten Theorie der Deichprofile, Ann. Phys., 2 (1809)
	412–445.
	
	

	
	\bibitem{cons3}
	 A. Constantin, J. Escher, Symmetry of steady periodic surface water waves with vorticity, J. Fluid Mech.,  498 (2004) 171–181.
	
		\bibitem{cons2}
	A. Constantin, M. Ehrnstr\"{o}m, E. Wahl\'{e}n, Symmetry of steady periodic gravity water waves with vorticity, Duke
	Math. J., 140 (2007) 591–603.
	
	\bibitem{cons4}
	 A. Constantin, R. Ivanov, C. Martin, Hamiltonian formulation for wave-current interactions in stratified rotational
	flows, Arch. Ration. Mech. Anal., 221 (2016) 1417–1447.
	
	\bibitem{cons2011}
	A. Constantin, Two-dimensionality of gravity water flows of constant nonzero vorticity beneath a surface wave
	train, Eur. J. Mech. B, Fluids, 30 (2011) 12–16.
	
	
		\bibitem{martin2018}
	C. Martin, Non-existence of time-dependent three-dimensional gravity water flows with
	constant non-zero vorticity, Phys. Fluids, 30 (2018)   107102.
	
	\bibitem{henry1}
	D. Henry, Equatorially trapped nonlinear waterwaves in a $\beta$-plane approximation with centripetal forces, J. Fluid
	Mech., 804 (2016) R1, 11 pp.
	
	
\bibitem{Wahlen}
M. Groves, E. Wahl$\acute{e}$n, Small-amplitude Stokes and solitary gravity water waves with an arbitrary
distribution of vorticity, Phys. D, 237 (2008) 1530-1538.
	
	\bibitem{henry2016}
	D. Henry, Equatorially trapped nonlinear water waves in a $\beta$-plane approximation with
	centripetal forces, J. Fluid Mech, (2016) 804 R1.
	
	\bibitem{wang1}
	J. Wang, M. Fe\v{c}kan, Y. Guan, Constant vorticity Ekman flows in the f-plane approximation, Discrete Cont Dyn-B,  27  (2022) 6619–6630.
	
	\bibitem{wang2}
	J. Wang, M. Fe\v{c}kan, Y. Guan, Constant vorticity Ekman flows in the $\beta$-plane approximation,  J. Math. Fluid Mech.,  23 (2021) 1-11.
	

	
	\bibitem{martin2023}
C. Martin, Liouville-type results for the time-dependent
three-dimensional (inviscid and viscous) water wave
problem with an interface, J. Differential Equations, 362 (2023) 88–105.

\bibitem{Henry2}
	D. Henry, G. Villari, Flow underlying coupled surface and internal waves, J. Differential Equations, 310 (2022) 404–442.
	
	
		\bibitem{c-e}
	J. Chu, J. Escher, Steady periodic equatorial water waves	with vorticity,  Discrete Contin. Dyn. Syst., 39 (2019) 4713-4729.
	
	\bibitem{chu3}
	J. Chu, X. Wang, L. Wang, Z. Zhang, A flow force reformulation of steady periodic fixed-depth irrotational equatorial flows, J. Differential Equations, 292 (2021) 220-246.
	
	
	
	\bibitem{IONESCU-KRUSE}
	D. Ionescu-Kruse, R. Ivanov, Nonlinear two-dimensional water waves with arbitrary
	vorticity, J. Differential Equations, 368 (2023) 317–349.
	
	\bibitem{z-d}
	G. Dai, Y. Zhang, Global bifurcation structure and some properties of
	steady periodic water waves with vorticity, J. Differential Equations, 349 (2023) 125–137.
	
	\bibitem{basu}
	B. Basu, C. Martin, Resonant interactions of rotational water waves in the equatorial f-plane approximation,
	J. Math. Phys., 59 (10) (2018) 103101, 9 pp.
	
	
	\end{thebibliography}
\end{document}